\documentclass[11pt]{amsart}
\usepackage{tabularx,booktabs,tikz}
\usepackage{caption}
\usepackage{amsmath}
\usepackage{amsfonts}

\usepackage{amscd}
\usepackage{amsthm}
\usepackage{amssymb} 
\usepackage{latexsym}
\usepackage{eufrak}
\usepackage{euscript}
\usepackage{epsfig}
\usepackage{graphics}
\usepackage{array}
\usepackage{enumerate}
\usepackage{dsfont}
\usepackage{color}
\usepackage{wasysym}
\usepackage{hyperref}
\usepackage{pdfsync}

\newcommand{\Hmm}[1]{\leavevmode{\marginpar{\tiny%
$\hbox to 0mm{\hspace*{-0.5mm}$\leftarrow$\hss}%
\vcenter{\vrule depth 0.1mm height 0.1mm width \the\marginparwidth}%
\hbox to
0mm{\hss$\rightarrow$\hspace*{-0.5mm}}$\\\relax\raggedright #1}}}

\newtheorem{theorem}{Theorem}[section]
\newtheorem{lemma}[theorem]{Lemma}
\newtheorem{corollary}[theorem]{Corollary}

\usepackage{natbib}
\begin{document}

\bibliographystyle{abbrv}

\title[Constrained variational problems on perturbed lattice graphs]{Constrained variational problems on perturbed lattice graphs}

\author{Weiqi Guan}
\address{Weiqi Guan: School of Mathematical Sciences, Fudan University, Shanghai 200433, China}
\email{wqguan24@m.fudan.edu.cn}

\begin{abstract}
	 In this paper, we solve some constrained variational problems on perturbed lattice graphs $G$. The first problem addresses the existence of ground state normalized solutions to Schr\"odinger equations
     \begin{equation*}
\left\{
    \begin{aligned}
        &-\Delta_{G} u+\lambda u=\vert u\vert^{p-2}u,x\in G\\
        &\Vert u\Vert_{l^2(G)}^2=a.
    \end{aligned}
    \right.    
\end{equation*}
We prove that if the graph is obtained by deleting finite edges in lattice graphs while maintaining connectivity, then there exists a threshold $\alpha_G\in[0,\infty)$ such that there do not exist ground state normalized solution if $0<a<\alpha_G$, and there exists a ground state normalized solution if $a>\alpha_G.$ If the graph is obtained by adding finite edges $E^{'}$ to lattice graphs, we prove that there exist $E^{'}$ and $a_1$ such that for all $a>a_1,$ there do not exist ground state normalized solutions.\\

The second problem concerns the existence of an extremal function for the Sobolev inequality. If the graph $G$ is obtained by deleting finite edges in lattice graphs while maintaining connectivity,  for the Sobolev super-critical regime, we prove that there exists an extremal function. for the Sobolev critical regime, we prove that there exists $G$ such that extremal can be attained. If the graph is obtained by adding finite edges $E^{'}$ to lattice graphs, we prove that there exists $E^{'}$ such that there does not exist an extremal function.
\end{abstract}
\par
\maketitle

\bigskip

\section{Introduction}\label{sec:intro}
Constrained variational problems arise in many problems. Recently, people are paying much attention to the normalized solutions to Schr\"odinger equations on $\mathbb{R}^d$; see e.g. \cite{MR3147450, MR4096725, MR4107073, MR4701352} and references therein. For the existence of normalized solutions to Schr\"odinger equations on graphs $G=(V, E).$
\begin{equation}\label{equa:normalized}
\left\{
    \begin{aligned}
        &-\Delta_{G} u+\lambda u=\vert u\vert^{p-2}u,x\in G\\
        &\Vert u\Vert_{l^2(G)}^2=a.
    \end{aligned}
    \right.    
\end{equation}
Where $\Delta_{G}$ is the Laplacian on $G$. The search for a normalized solution to problem (\ref{equa:normalized}) can be reduced to finding a critical point of the functional
\[
\Phi_{G}^p(u) = \frac{1}{2}\int_{G}\vert\nabla_G u\vert_2^2dx-\frac{1}{p}\int_{G}\vert u\vert^pdx
\]
over the sphere $S^a_G:=\{u\in l^2(G):\Vert u\Vert_{l^2(G)}^2=a\}$. See section \ref{sec:pre} for definitions. If $\Phi_{G}^p(u)$ is bounded from below on $S^a_G$, we set
\begin{equation}\label{pro:norm}
E_{G}^{a,p}:=\inf_{u\in S^a_G}\Phi_{G}^p(u).    
\end{equation}

 When $G$ is a finite graph, Yang et al. \cite{YANG2024128173} proved that $E_{G}^{a,p}$ is attained for all $a>0$. When $G$ is the lattice graph $\mathbb{Z}^d$, M. I. Weinstein \cite{MR1690199} used discrete concentration compactness principle to prove the existence of a critical parameter $\alpha_{\mathbb{Z}^d}$ such that $E_{\mathbb{Z}^d}^{a,p}$ is attainable when $0 < a<\alpha_{\mathbb{Z}^d}$, but becomes unattainable for $a >\alpha_{\mathbb{Z}^d}$. Mass critical index $2+\frac{4}{d}$ plays a crucial role in depending whether $\alpha_{\mathbb{Z}^d}>0$. M. I. Weinstein further showed that $\alpha_{\mathbb{Z}^d}=0$ if $2<p<2+\frac{4}{d}$ and $\alpha_{\mathbb{Z}^d}>0$ if $p\geq 2+\frac{4}{d}$. The detailed proof in the case $2<p<2+\frac{4}{d}$ is given in \cite{MR4608774}. Afterwards, the author \cite{guan2025existencenormalizedsolutionsnonlinear} extended M. I. Weinstein's result to Schr\"odinger equations with nonlinearity satisfying Berestycki-Lions type condition and potential which is a trapping potential or well potential. \\

Another constrained problem is the existence of extremal functions for Sobolev inequalities on $\mathbb{R}^d$. For $d>1$, the classical concentration-compactness principle established by P.L.Lions \cite{MR778970, MR778974, MR834360, MR850686} can be applied to find the extremal functions. For the noncompact, complete Riemannian manifold $M^{d}$ with $d\geq2$, non-negative Ricci curvature and Euclidean volume growth, the sharp Sobolev inequality is established in \citep{balogh2023sharp} for $1<p< d$
\[
\Vert u\Vert_{L^{p^{*}}(M^d)}\leq AVG(M^d)^{-\frac{1}{d}}S_{d,p}\Vert\nabla u\Vert_{L^p(M^d)},
\]
where 
\[
AVG(M^d) = \lim_{R\rightarrow\infty}\frac{Vol_{g}(B_R(x))}{\omega_dR^d}>0
\]
is the Euclidean volume growth condition, $S_{d,p}$ is the best Sobolev constant in Euclidean space. By Bishop-Gromov comparison theorem $AVG(M^d)\in(0,1]$. It is proved in \cite{nobili2025fine} that there exists an extremal function if and only if $M^d$ is isometric to Euclidean space. Since $AVG(M^d)=1$ if and only if $M^d$ is isometric to Euclidean space \cite{colding1997ricci,de2018non}. Hence if $0<AVG(M^d)<1$, there does not exist an extremal function. Consider the  existence of extremal functions for discrete Sobolev inequalities on lattice graphs
\[
S_{\mathbb{Z}^d}^{p,q}\Vert u\Vert_{l^q(\mathbb{Z}^d)}\leq \Vert u\Vert_{D^{1,p}(\mathbb{Z}^d)},
\]
where $d\geq 3$, $1\leq p< d,$$q\geq\frac{dp}{d-p}$, $u\in D^{1,p}(\mathbb{Z}^d)$, $\Vert u\Vert_{D^{1,p}(\mathbb{Z}^d)}=(\int_{\mathbb{Z}^d}\vert\nabla_{\mathbb{Z}^d} u\vert_p^pdx)^{\frac{1}{p}}$ and $D^{1,p}(\mathbb{Z}^d)$ is the completion of $C_0(\mathbb{Z}^d)$ under the $D^{1,p}(\mathbb{Z}^d)$ norm; see \cite{HUA20156162} for more details about discrete Sobolev inequality. This problem reduces to studying the attainability of the variational problem:
\begin{equation*}
    J_{\mathbb{Z}^d}^{a,p,q} = \inf\limits_{u\in l^q(\mathbb{Z}^d),\Vert u\Vert_{l^q(\mathbb{Z}^d)}^q=a}\int_{\mathbb{Z}^d}\vert\nabla_{\mathbb{Z}^d} u\vert^p_pdx.
\end{equation*}

Bobo Hua et al. \cite{MR4328634} used the discrete concentration compactness principle to prove that $J_{\mathbb{Z}^d}^a$ is attained if $q>\frac{dp}{d-p}$. The Sobolev critical case $q=\frac{dp}{d-p}$ remains open. For the general infinite, locally finite graph, a natural question arise about the validity of Sobolev inequality and the existence of extremal function. Similar to works on noncompact manifold, the Sobolev inequality and the existence should concerns with the curvature or structure of graphs. In this paper we consider the perturbation the lattice graphs $\mathbb{Z}^d=(V,E)$ to a connected graph $G=(V,\tilde{E})$, where $\tilde{E}= E\backslash E^{'}$ or $\tilde{E}= E\cup E^{'}$, $E^{'}$ is composed of a finite set of edges. The reason for considered perturbation is that deleting edges make the graph less connected than lattice graph and hence it is easily to check that the resulting graph $G=(V,E\backslash E^{'})$ satisfying non-positive Ollivier-Ricci curvature condition $Ric_{ij}\leq 0$ for all edge $(i,j)\in E\backslash E^{'}$. In contrast, adding edges make the graph more connected than lattice graph and hence the resulting graph $G=(V,E\cup  E^{'})$ satisfying curvature condition $Ric_{ij}\geq 0$ for all edge $(i,j)\in E$. The discrete Sobolev inequalities also hold on $G$

\begin{equation}\label{Soboleb G}
S_{G}^{p,q}\Vert u\Vert_{l^q(G)}\leq\Vert u\Vert_{D^{1,p}(G)},    
\end{equation}
with the same conditions on $p,q$ and similar definition for $D^{1,p}(G).$ We refer to Lemma \ref{sobolev} for details. Similarly, finding the extremal function to (\ref{Soboleb G}) reduces to studying the attainability of the variational problem:
\begin{equation}\label{pro:sobolev}
J_{G}^{a,p,q} = \inf\limits_{u\in l^q(G),\Vert u\Vert_{l^q(G)}^q=a}\int_{G}\vert\nabla_{G} u\vert^p_pdx.    
\end{equation}

 However, in the setting of the lattice graphs, the proofs of the problems (\ref{pro:norm})(\ref{pro:sobolev}) are essentially based on the translation invariance of functionals over the lattice graphs. Take the problem (\ref{pro:sobolev}) for example. Consider the minimizing sequence $\{u_n\}$ such that $u_n\rightharpoonup u$ in $l^{p}(\mathbb{Z^d})$.
 Then we need to exclude the case $u=0$. Excluding the vanishing case in concentration compactness principle  $\limsup\limits_{n\rightarrow\infty}\Vert u_n\Vert_{l^{\infty}(\mathbb{Z}^d)}>0$ alone can not ensure that $u\neq 0$. A simple example is the minimizing sequence $\{u(x+x_n)\}$ with $x_n\rightarrow\infty$ and $u$ the extremal function. Hence we need to translate the point $x_n$ achieving maximum of $u_n$ to the original point and consider sequence $\{u_n(x+x_n)\}$, then by translation invariance $\int_{\mathbb{Z}^d}\vert\nabla_{\mathbb{Z}^d} u_n(x+x_n)\vert^p_pdx=\int_{\mathbb{Z}^d}\vert\nabla_{\mathbb{Z}^d} u_n(x)\vert^p_pdx$, $\{u_n(x+x_n)\}$ is still minimizing sequence. On the general infinite graphs, the functionals lack translation invariance
 \[
 \int_{G}\vert\nabla_{G} u_n(x+x_n)\vert^p_pdx\neq\int_{G}\vert\nabla_{G} u_n(x)\vert^p_pdx.
 \] 
 This difficulty hinders us from extending results to general infinite graphs and we need to develop a new method to overcome this difficulty. In this paper, we develop a comparison method and apply this method to constrained problems (\ref{pro:norm})(\ref{pro:sobolev}) on the perturbed lattice graphs. The paradigm of this method is hopefully to be widely applied to other problems such as the existence of ground state solution. The following is our first main result.
\begin{theorem}\label{Thm 1}
    Let $G=(V,E\backslash {E}^{'})$ be a connected graph obtained by deleting finite edges $\mathcal{E}^{'}$ in lattice graphs $\mathbb{Z}^d$ but keep the connectivity. Then we have\\
    (1)If $d\geq 1$ and $p>2$, then there exists a threshold $\alpha_G$ such that $E_{G}^{a,p}$ can be attained if $0< a <\alpha_G$, and $E^{a,p}_G$ can not be attained if $a>\alpha_G$.\\
    (2)If $d\geq 3$, $1\leq p<d$ and $q>\frac{dp}{d-p}$, then for all $a>0$, $J_{G}^{a,p,q}$ can be attained.\\
    (3)If $d\geq 3$, $1\leq p<d$, $q=\frac{dp}{d-p}$ and for some $a_0>0$, $J^{a_0,p,q}_{G}<J^{a_0,p,q}_{\mathbb{Z}^d}$, then for all $a>0$, $J^{a,p,q}_G$ can be attained. Moreover, there exists $G$ satisfying conditions.
\end{theorem}
     Theorem \ref{Thm 1} (1)(2) extends results of \cite{MR1690199,MR4328634} to $G=(V,E\backslash {E}^{'})$. Since $G=(V,E\backslash {E}^{'})$ satisfying non-positive Ollivier-Ricci curvature condition, Theorem \ref{Thm 1} (2)(3) is a opposite result in a non-positive Ollivier-Ricci curvature graph comparing to nonexistence of extremal function in non-negative Ricci curvature manifold\cite{nobili2025fine}. Also since \cite{MR4328634} only resolved the regime $(\frac{dp}{d-p},\infty)$, the Sobolev critical case $q=\frac{dp}{d-p}$ still remains open in lattice graphs. Theorem \ref{Thm 1} (3) first give the existence of extremal function for the Sobolev critical exponent in graphs $G=(V,E\backslash {E}^{'})$ satisfying $J^{a_0,p,q}_{G}<J^{a_0,p,q}_{\mathbb{Z}^d}$. A natural question arises whether Theorem \ref{Thm 1} holds for general infinite graphs or at least for arbitrary perturbations of the lattice graphs. In the next main theorem, we show that this is not true.
\begin{theorem}\label{Thm 2}
    Let $G=(V,E\cup E^{'})$ be a graph obtained by adding finite edges $E^{'}$ in lattice graphs $\mathbb{Z}^d$. Then we have\\
    (1) If $d\geq 1$ and $p>2$, then there exist $E^{'}$ and $a_1$ such that $E_{G}^{a,p}$ can not be attained for $a> a_1$.\\
    (2) If $d\geq 3$, $1\leq p<d$ and $q\geq\frac{dp}{d-p}$, then there exists a set $E^{'}$ such that $J_{G}^{a,p,q}$ can not be attained for all $a>0$.
\end{theorem}
Theorem \ref{Thm 2} can be seen as a generalization of the nonexistence of extremal function \cite{nobili2025fine} to the discrete setting. This paper is organized as follows: In section \ref{sec:pre}, we recall basic definitions and give some useful lemmas. In section \ref{sec:delete}, we prove Theorem \ref{Thm 1}. In section \ref{sec:addingedges}, we prove Theorem \ref{Thm 2}.
\section{Preliminaries}\label{sec:pre}
In this section, we recall the basic setting for analysis on undirected connected locally finite graphs. Let $G=(V, E)$ be an undirected connected locally finite graph, where $V$ is a countable set of vertices, and $E$ is the edge set. We denote by $x\sim y$ if $(x,y)\in E$. Denote $C(G)$ as functions defined on $V$ and $C_0(G)$ as functions with finite support. Then for a function $u\in C(G),$ the Laplacian $\Delta_G $ is defined as follows.
\[
\Delta_G u(x)= \sum_{y\sim x}(u(y)-u(x)).
\]
The $l^p$ summable function space on the lattice is defined as 
\[
l^p(G):=\{u\in C(G):\Vert u\Vert_{l^p(G)}<\infty\},
\]
where $p\in [1,\infty]$ and $l^p$ norm of a function $u\in C(G)$ is defined as 
\begin{equation*}
\Vert u\Vert_{l^p(G)}:=\left\{
\begin{aligned}
    &(\sum_{x\in G}\vert u(x)\vert^p)^{\frac{1}{p}},1\leq p\textless\infty\\
    &\sup_{x\in G}\vert u(x)\vert,p=\infty.
\end{aligned}
\right.
\end{equation*}
When $u\in l^1(G),$ the integration of $u$ over $G$ is defined as
\[
\int_{G}udx=\sum_{x\in V}u(x).
\]
For any $u\in C(G),$ the $p-$norm of the gradient of $u$ at $x$ is defined as 
\[
\vert \nabla_{G} u\vert_p(x)=(\frac{1}{2}\sum_{y\sim x}\vert u(y)-u(x)\vert^p)^{\frac{1}{p}}.
\]
Furthermore, $D^{1,p}$ norm is defined as $D^{1,p}(u)=(\int_{G}\vert\nabla_G u\vert_p^pdx)^{\frac{1}{p}}$, and $D^{1,p}(G)$ is the completion of the $C_0(G)$ under $D^{1,p}(G)$ norm. Lattice graphs $\mathbb{Z}^d=(V,E)$ is composed of vertices $V=\{x=(x_1,...,x_d):x_i\in\mathbb{Z},1\leq i\leq d\}$ and edges $E=\{(x,y):\sum\limits_{1\leq i\leq d}\vert x_i-y_i\vert=1\}$. $\pm e_k$ denote vertex $(0,...,\pm1,...,0)$ where only $k$-th component is $\pm1$ and other components are $0$. For any $x,y\in\mathbb{Z}^d$, we define $x+y:=(x_1+y_1,...,x_d+y_d)$. The ball $B_R(x)$ denotes the set $\{y:\vert y-x\vert_{\infty}<R\}$, where $\vert y-x\vert_{\infty}=\max\limits_{1\leq i\leq d}\vert y_i-x_i\vert.$ We abbreviate $B_R(0)$ as $B_R$ for simplicity. The boundary of the ball $B_R(x)$ is defined as $\partial B_R(x)=\{(y,z):y\in B_R(x), z\notin B_R(x)\}$First, we recall the basic lemma; see \citep{MR4328634} for this lemma in the setting of the lattice graphs.
\begin{lemma}\label{normcontrol}
For $u\in l^{p}(G)$, we have $u\in l^{q}(G)$ for all $q\textgreater p$ and $\Vert u\Vert_{l^q(G)}\leq\Vert u\Vert_{l^p(G)}$.
\end{lemma}
\begin{proof}
    We directly compute that
    \begin{align*}
        \Vert u\Vert_{l^q(G)}^q&=\sum_{x\in G}\vert u(x)\vert^q\\
        &\leq\sum_{x\in G}\vert u(x)\vert^p\vert u\vert_{l^{\infty}(G)}^{q-p}\\
        &\leq (\sum_{x\in G}\vert u(x)\vert^p)^{\frac{q}{p}}\\
        &=\Vert u\Vert_{l^p(G)}^{q}.
    \end{align*}
\end{proof}
Using the above lemma, we prove that $\Phi_G^p(u)$ is bounded from below on the sphere.
\begin{lemma}
    Let $G=(V, E)$ be a connected locally finite graph, then $E^a_G>-\infty$ for all $a>0$.
\end{lemma}
\begin{proof}
    For all $u\in S^a_G$, $\Phi^p_G(u)\geq -\frac{1}{p}\int_G\vert u\vert^pdx$. Also, by Lemma \ref{normcontrol}, $\int_G\vert u\vert^pdx\leq (\int_G\vert u\vert^2dx)^{\frac{p}{2}}=a^{\frac{p}{2}}$. Hence $\Phi^p_G(u)\geq -\frac{1}{p}a^{\frac{p}{2}}$ and we have the result.
\end{proof}
After we recall the classical Brezis-Lieb lemma \citep{MR699419}.
\begin{lemma}\label{BreLieb}
    Let $(\Omega,\Sigma,\tau)$ be a measure space, where $\Omega$ is a set equipped with a $\sigma-$algebra $\Sigma$ and a Borel measure $\tau:\Sigma\rightarrow[0,\infty]$. Given a sequence $\{u_n\}\subset L^p(\Omega,\Sigma,\tau)$ with $0\textless p\textless\infty$. If $\{u_n\}$ is uniformly bounded in $L^p(\Omega,\Sigma,\tau)$ and $u_n\rightarrow u$, $\tau$-a.e. in $\Omega$, then we have that
    \[
    \lim\limits_{n\rightarrow\infty}(\Vert u_n\Vert_{L^p(\Omega)}^p-\Vert u_n-u\Vert_{L^p(\Omega)}^p)=\Vert u\Vert_{L^p(\Omega)}^p.
    \]
\end{lemma}
A direct consequence of Lemma \ref{BreLieb} is the following.
\begin{corollary}\label{NormBreLieb}
For a bounded sequence $\{u_n\}\subset l^p(G)$ such that $u_n\rightharpoonup u$ in $l^p(G)$ with $p\geq 1$, then for any $q\geq p$ we have\\
    (1)\[
    \lim\limits_{n\rightarrow\infty}(\Vert u_n\Vert_{l^q(G)}^q-\Vert u_n-u\Vert_{l^q(G)}^q)=\Vert u\Vert_{l^q(G)}^q,
    \]
    (2)
    \[
    \lim\limits_{n\rightarrow\infty}(\Vert u_n\Vert_{D^{1,q}(G)}^q-\Vert u_n- u\Vert_{D^{1,q}(G)}^q)=\Vert u\Vert_{D^{1,q}(G)}^q.
    \]
\end{corollary}
\begin{proof}
    Since $u_n\rightharpoonup u$ in $l^p(G)$, for any $x\in G$ take delta function 
    \begin{equation*}
\delta_x(y)=\left\{
    \begin{aligned}
        &1,y=x \\
        &0, \mathrm{otherwise}
    \end{aligned}
    \right.    
\end{equation*}
Then the delta function belongs to the dual space of $l^p(G).$ Take delta function as test function we have that $\lim\limits_{n\rightarrow\infty}u_n(x)=u(x)$ for all $x\in G.$ Now for $q\geq p,$ since $\{u_n\}$ is bounded in $l^p(G)$, by Lemma \ref{normcontrol} $\{u_n\}$ is also bounded in $l^q(G)$. Hence by Lemma \ref{BreLieb} (1) is proved.\\

For (2) we consider measure space $(E,\Sigma,\tau)$, where $\Sigma$ is generated by $\{(x,y)\in E\}$ and $\tau((x,y))=1.$ Define $\vert \nabla_G g\vert((x,y)):=\vert g(x)-g(y)\vert$ over the measure space for any function $g$. Then we have
\begin{align*}
    &\Vert \nabla _Gu_n\Vert_{l^q(E)}^q\\
    &=\sum_{(x,y)\in E}\vert u_n(x)-u_n(y)\vert^q\\
    &\leq C\sum_{x\in V}\vert u_n\vert^q\\
    &\leq C\Vert u_n\Vert_{l^p(G)}^q\\
    &\leq C.
\end{align*}
Hence $\{\vert\nabla_G u_n\vert\}$ is bounded in $l^{q}(E).$ Also as the proof in (1), $\lim\limits_{n\rightarrow\infty}u_n(x)=u(x)$ for all $x\in V.$ Hence $\lim\limits_{n\rightarrow\infty}\vert u_n(x)-u_n(y)\vert=\vert u(x)-u(y)\vert$ for all $(x,y)\in E.$ Then by Lemma \ref{BreLieb} we have the desired result.
    
\end{proof}
Next, we recall the discrete Sobolev inequality on lattice graphs; see e.g. \cite{HUA20156162, MR4328634}.
\begin{lemma}
    Suppose $d\geq 3$, $1\leq p< d,$$q\geq\frac{dp}{d-p}$, then there exists $S^{p,q}_{\mathbb{Z}^d}$ such that for all $u\in D^{1,p}(\mathbb{Z}^d)$, we have
    \[
    S^{p,q}_{\mathbb{Z}^d}\Vert u\Vert_{l^q(\mathbb{Z}^d)}\leq \Vert u\Vert_{D^{1,p}(\mathbb{Z}^d)}.
    \]
\end{lemma}
In this paper, we consider two types of graphs $G=(V,\tilde{E})$ obtained by perturbations of the lattice graph $\mathbb{Z}^d=(V, E)$. The first one is $G=(V,E\backslash E^{'})$ which is obtained by deleting edges $E^{'}\subset E$ such that $x,y\in B_R$ for all $(x,y)\in E^{'}$ and some $R>0$ while maintaining $G$ connected. Another is $G=(V,E\cup E^{'})$ which is obtained by adding edges $E^{'}$ such that $x,y\in B_R$ for all $(x,y)\in E^{'}$ and some $R>0$. Since we do not change the node set $V$, we abuse $x\in\mathbb{Z}^d$ to denote $x\in V$ for clarity. We prove the discrete Sobolev inequality for $G$ as follows.
\begin{lemma}\label{sobolev}
    Let $G=(V,\tilde{E})$ be a graph perturbed from the lattice graph $\mathbb{Z}^d$ as above, where $d\geq 3$, $1\leq p< d$, $q\geq\frac{dp}{d-p}$, $u\in D^{1,p}(G)$. Then there exists a constant $S^{p,q}_G$ such that
    \[
    S^{p,q}_G\Vert u\Vert_{l^q(G)}\leq\Vert u\Vert_{D^{1,p}(G)}.
    \]
\end{lemma}
\begin{proof}
    We have 
    \[
    \Vert u\Vert_{l^q(G)}=\Vert u\Vert_{l^q(\mathbb{Z}^d)}\leq C(p,q)\Vert u\Vert_{D^{1,p}(\mathbb{Z}^d)}.
    \]
    If $G$ is obtained by adding edges within $B_R$, then apparently we have 
    \[
    \Vert u\Vert_{D^{1,p}(\mathbb{Z}^d)}\leq \Vert u\Vert_{D^{1,p}(G)}.
    \]
    Which yields the result. If $G$ is obtained by deleting edges $E^{'}$ within $B_R$, but keeping connectivity. Given $(x,y)\in E^{'},$ suppose a path $(x_0=x,x_1,...,x_k=y)$ such that $x_{i}\sim x_{i+1}$ in $G.$ Then we have
    \begin{align*}
        \vert u(x)-u(y)\vert^p\leq C(G,p)\sum_{i=0}^{k-1}\vert u(x_i)-u(x_{i+1})\vert^p\leq C(G,p)\Vert u\Vert_{D^{1,p}(G)}^p .
    \end{align*}
    This proves the result.
\end{proof}
\section{Perturbation by deleting edges}\label{sec:delete}
In this section, we suppose $G=(V, E\backslash E^{'})$, which is a graph obtained by deleting edges $E^{'}$ in lattice graphs $\mathbb{Z}^d=(V, E)$ within the ball $B_R$ while keeping the connectivity. When we consider the attainability of $E^{a,p}_G$, we always assume $p>2$. When we consider the attainability of $J^{a,p,q}_G$, we assume $d\geq 3$ further. First, before the proof of (1) in Theorem \ref{Thm 1}, we recall the existence of a threshold for the attainability of $E^{a,p}_{\mathbb{Z}^d}$; see e.g. \cite{MR1690199,guan2025existencenormalizedsolutionsnonlinear}.
\begin{lemma}\label{lemma:thresholdlattice}
    (1)Suppose $d\geq 1$ and $p>2$. Then there exists threshold $\alpha_{\mathbb{Z}^d}\in[0,+\infty)$ such that if $0<a<\alpha_{\mathbb{Z}^d}$, then $E^{a,p}_{\mathbb{Z}^d}$ can not be attained, if $a>\alpha_{\mathbb{Z}^d}$, then $E^{a,p}_{\mathbb{Z}^d}$ can be attained.\\ 

    (2) If $E^{a,p}_{\mathbb{Z}^d}<0$, then $E^{a,p}_{\mathbb{Z}^d}<0$ can be attained.
\end{lemma}
Now we need to establish the property of the function $a\mapsto E_G^{a,p}$. The proof of the following lemma follows the same verbatim as \cite[Lemma 4.2]{guan2025existencenormalizedsolutionsnonlinear}; see also \cite{MR4390628}. So we omit the proof here.
\begin{lemma}\label{funproper}
     (1) $E^{a,p}_G\leq 0$ for all $a>0$.\\
     (2) $E^{a+b,p}_G\leq E^{a,p}_G+E^{b,p}_G$ for $a,b\textgreater0$.\\
     (3) If $E^{a,p}_G$ or $E^{b,p}_G$ is attained, then $E^{a+b,p}_G\textless E^{a,p}_G+E^{b,p}_G$.\\
     (4) Function $a\mapsto E^{a,p}_G$ is continuous and non-increasing.
\end{lemma}
We also need the following useful lemmas.
\begin{lemma}\label{lemma:normalizedless}
    (1) $E^{a,p}_G\leq E^{a,p}_{\mathbb{Z}^d}$.\\
    (2) If $E^{a,p}_{\mathbb{Z}^d}$ is attained, then $E^{a,p}_G\textless E^{a,p}_{\mathbb{Z}^d}$.\\
    
\end{lemma}
\begin{proof}
    (1) This follows from $\Phi^p_G(u)=\Phi_{\mathbb{Z}^d}^p(u)-\frac{1}{2}\sum\limits_{(x,y)\in E^{'}}\vert u(x)-u(y)\vert^2$ for any $u\in S^a_{G}$.\\
    (2) Suppose $E^{a,p}_{\mathbb{Z}^d}$ is attained by $u\in S^a_{\mathbb{Z}^d}$. Let $(z,w)\in E^{'}$, without loss of generality, we assume $w= z+e_1$. We claim that there exists $y\in \mathbb{Z}^d$ such that $u(y)\neq u(y+e_1)$. If this is not true, for any $x\in\mathbb{Z}^d$ and any $k\in \mathbb{Z}$, $u(x)=u(x+ke_1)$. Since $\Vert u\Vert_{l^2(\mathbb{Z}^d)}\textless\infty$, letting $k\rightarrow \infty$, we have $u(x+ke_1)\rightarrow0$. Hence $u(x)=0$ for any $x\in\mathbb{Z}^d$. But this is in contradiction to $\Vert u\Vert_{l^2(\mathbb{Z}^d)}^2=a$. Set $\tilde{u}(x)= u(x-z+y)$, then $\tilde{u}(z)=u(y)$ and we have $\vert\tilde{u}(z)-\tilde{u}(z+e_1)\vert\neq 0$. Hence we have
    \begin{align*}
    E^{a,p}_G&\leq\Phi_G^p(\tilde{u})\\
    &=\Phi_{\mathbb{Z}^d}^p(\tilde{u})-\frac{1}{2}\vert\tilde{u}(z)-\tilde{u}(z+e_1)\vert^2    \\
    &=E^{a,p}_{\mathbb{Z}^d}-\frac{1}{2}\vert\tilde{u}(z)-\tilde{u}(z+e_1)\vert^2\\
    &\textless E^{a,p}_{\mathbb{Z}^d}.
    \end{align*}
\end{proof}
\begin{lemma}\label{lemma:normalizedcompare}
    If $E^{a,p}_G\textless E^{a,p}_{\mathbb{Z}^d}$, then $E^{a,p}_{G}$ is attained.
\end{lemma}
\begin{proof}
    Let $\{u_n\}\subset S^a_G$ be a minimizing sequence, taking a subsequence if necessary, we assume $u_n\rightharpoonup u$ in $l^2(G)$, we claim that $\Vert u\Vert_{l^2(G)}^2= a$. \\
    
    If the claim is not true. If $u=0$, then for any $x\in V$, $\lim\limits_{n\rightarrow\infty}u_n(x)=0$. For any $\epsilon\textgreater0$, take a sufficiently large $n$ such that for all $(x,y)\in E^{'}$, we have $\vert u(x)\vert \leq \epsilon$ and $\vert u(y)\vert\leq\epsilon$. Also we take a sufficiently large $n$ such that $\Phi_G^p(u_n)\leq E^{a,p}_G+\epsilon$. On the one hand, we have 
    \begin{align*}
        \Phi_{\mathbb{Z}^d}^p(u_n)&=\Phi_G^p(u_n)+\frac{1}{2}\sum\limits_{(x,y)\in \mathcal{E}^{'}}\vert u_n(x)-u_n(y)\vert^2\\
        &\leq E^{a,p}_G+ C\epsilon.\\
    \end{align*}
    On the other hand we have $E^{a,p}_{\mathbb{Z}^d}\leq \Phi^p_{\mathbb{Z}^d}(u_n)$, take $\epsilon$ sufficiently small will lead to a contradiction to $E^{a,p}_G\textless E^{a,p}_{\mathbb{Z}^d}$.\\
    
    If $0\textless\Vert u\Vert_{l^2(G)}^2\textless a$, set $\beta = \Vert u\Vert_{l^2(G)}^2$, then by Corollary \ref{NormBreLieb} we have $\Vert u_n-u\Vert_{l^2(G)}^2 = a-\beta+o(1)$, $ \Vert \nabla_G u_n\Vert_{l^2(G)}^2-\Vert \nabla_G (u_n-u)\Vert_{l^2(G)}^2= \Vert \nabla_G  u\Vert_{l^2(G)}^2+o(1)$, $\Vert u_n\Vert_{l^p(G)}^p-\Vert u_n-u\Vert_{l^p(G)}^p=\Vert  u\Vert_{l^p(G)}^p+o(1)$. If $E^{\beta,p}_G$ is not attained by $u$, then we have 
    \begin{align*}
        E^{\beta,p}_{G}&< \Phi^p_G(u)\\
        &=\Phi^p_G(u_n)-\Phi^p_G(u_n-u)+o(1)\\
        &\leq E^{a,p}_{G} - E^{a-\beta,p}_{G}+o(1).
    \end{align*}
    Taking $n\rightarrow\infty$ we have $E^{\beta,p}_G<E^{a,p}_{G} - E^{a-\beta,p}_{G}.$ But this is in contradiction to $E^{\beta,p}_G+ E^{a-\beta,p}_G\geq E^{a,p}_G$. If $E^{\beta,p}_G$ is attained by $u$, then with similar argument we have
    \[
    E^{\beta,p}_{G}\leq  E^{a,p}_{G} - E^{a-\beta,p}_{G}.
    \]
    Since $E^{\beta,p}_G$ is attained, we have $E^{\beta,p}_G+ E^{a-\beta,p}_G> E^{a,p}_G$ by (3) in Lemma \ref{funproper}. This is a contradiction.\\
    
    In conclusion, $\Vert u\Vert_{l^2(G)}^2=a$. By Lemma \ref{NormBreLieb} we have $u_n\rightarrow u$ in $l^2(G)$ as $n\rightarrow\infty$. Hence by $\Vert\nabla_G(u_n-u)\Vert_{l^2(G)}^2\leq C\Vert u_n-u\Vert_{l^2(G)}^2$ and $\Vert u_n-u\Vert_{l^p(G)}^p\leq \Vert u_n-u\Vert_{l^2(G)}^p$, we have $\lim\limits_{n\rightarrow\infty}\Vert \nabla u_n\Vert_{l^2(G)}^2=\Vert \nabla u\Vert_{l^2(G)}^2$ and  $\lim\limits_{n\rightarrow\infty}\Vert  u_n\Vert_{l^p(G)}^p=\Vert u\Vert_{l^p(G)}^p$ by Corollary \ref{NormBreLieb}. Hence $\Phi_G^p(u)= \lim\limits_{n\rightarrow\infty}\Phi^p_{G}(u_n)=E^{a,p}_G$.
\end{proof}
Now we can find the threshold for the problem.
\begin{theorem}
    Let $\alpha_G=\inf\{a\in(0,\infty):E^{a,p}_G< 0\}$, then if $0< a< \alpha_G$, then $E^a_G$ can not be attained, if $a>\alpha_G$, then $E^a_G$ can be attained.
\end{theorem}
\begin{proof}
    If $0< a< \alpha_G$ and $E^{a,p}_G$ can be attained, then $E^{\alpha_G,p}_G< E^{a,p}_G\leq 0$, this is in contradiction to $E^{\alpha_G,p}_G=0$.\\
    
    If $a>\alpha_G$. If $E^{a,p}_{\mathbb{Z}^d}=0$, then $E^{a,p}_{G}< E^{a,p}_{\mathbb{Z}^d}=0$, by Lemma \ref{lemma:normalizedcompare} $E^{a,p}_{G}$ is attained. If $E^{a,p}_{\mathbb{Z}^d}< 0$, then by Lemma \ref{lemma:thresholdlattice} $E^a_{\mathbb{Z}^d}$ can be attained. By (2) of Lemma \ref{lemma:normalizedless} we have $E^{a,p}_{G}< E^{a,p}_{\mathbb{Z^d}}$, by Lemma \ref{lemma:normalizedcompare} again $E^{a,p}_G$ can be attained.
\end{proof}
This completes the proof for (1) in Theorem \ref{Thm 1}. Now we turn to the proof of (2) in Theorem \ref{Thm 1}. First, we establish properties for the function $a\mapsto J^{a,p,q}_{G}.$ The proof is similar to Lemma \ref{funproper}.
\begin{lemma}\label{sobolevfunproperty}
     (1) $J^{\theta a,p,q}_G<\theta J^{a,p,q}_G$ for all $\theta>1$.
     (2) $J^{a+b,p,q}_G< J^{a,p,q}_G+J^{b,p,q}_G$ for $a,b>0$.\\
     (3) Function $a\mapsto J^{a,p,q}_G$ is continuous.
     
\end{lemma}
\begin{proof}
    (1)Let $u\in l^{q}(G)$ such that $\Vert u\Vert_{l^q(G)}^q=a$ and $\int_{G}\vert\nabla_G u\vert^p_pdx\leq J^{a,p,q}_G+\epsilon$. Then for $\theta>1$, we have $\Vert \theta^{\frac{1}{q}} u\Vert_{l^q(G)}^q=\theta\Vert u\Vert_{l^q(G)}^q=\theta a.$ Hence, we have
    \begin{align*}
        J^{\theta a,p,q}_G&\leq\theta^{\frac{p}{q}}\int_{G}\vert\nabla_G u\vert^pdx\\
        &\leq \theta^{\frac{p}{q}}(J^{a,p,q}_G+\epsilon).
    \end{align*}
    Since $\epsilon>0$ is arbitrary and $J^{a,p,q}_G>0$, we have 
    \[
    J^{\theta a,p,q}_G\leq \theta^{\frac{p}{q}}J^{a,p,q}_G<\theta J^{a,p,q}_G.
    \]

    (2)Without loss of generality, we assume $a>b.$ Then by (1), we have
    \begin{align*}
        J^{a+b,p,q}_G&=J^{a(1+\frac{b}{a}),p,q}_G\\
        &< (1+\frac{b}{a})J^{a,p,q}_G\\
        &=J^{a,p,q}_G+\frac{b}{a}J^{\frac{a}{b}b,p,q}_G\\
        &<J^{a,p,q}_G+J^{b,p,q}_G.
    \end{align*}

    (3) Given a sequence $\{a_n\}$ such that $\lim\limits_{n\rightarrow\infty}a_n=a$. Take $u_n\in l^q(G)$ such that $\Vert u_n\Vert_{l^q(G)}^q=a_n$ and $\int_{G}\vert\nabla_G u\vert^p_pdx\leq J^{a_n,p,q}_G+\frac{1}{n}$. Set $v_n=(\frac{a}{a_n})^{\frac{1}{q}}u_n$, then $\Vert v_n\Vert_{l^q(G)}^q=a$. We compute that
    \begin{align*}
        &\vert\int_{G}\vert\nabla_G u_n\vert^p_pdx-\int_{G}\vert\nabla_G v_n\vert^p_pdx\vert\\
        &=\vert 1-(\frac{a}{a_n})^{\frac{p}{q}}\vert\int_{G}\vert\nabla_G u_n\vert^p_pdx\\
        &\leq C\vert a_n^{\frac{p}{q}}-a^{\frac{p}{q}}\vert.
    \end{align*}
    Hence $J^{a,p,q}_{G}\leq\int_{G}\vert\nabla_G v_n\vert^p_pdx=\int_{G}\vert\nabla_G u_n\vert^p_pdx+o(1)=J^{a_n}_G+o(1)$. The argument to $J^{a_n,p,q}_G\leq J^{a,p,q}_{G}+o(1)$ is similar. 
\end{proof}
Next, we establish a useful lemma; the proof is similar to Lemma \ref{lemma:normalizedcompare}.
\begin{lemma}\label{lemma:sobolevcompare}
    If $J^{a,p,q}_G\textless J^{a,p,q}_{\mathbb{Z}^d}$, then $J^{a,p,q}_{G}$ is attained.
\end{lemma}
\begin{proof}
    Let $\{u_n\}\subset \{u\in l^q(G):\Vert u\Vert^q_{l^q(G)}=a\}$ be a minimizing sequence with respect to $J^{a,p,q}_G$. Taking a subsequence if necessary, we assume $u_n\rightharpoonup u$ in $l^q(G)$, we claim that $\Vert u\Vert_{l^q(G)}^q=a.$\\

    If the claim is not true. If $u=0$, then $\lim\limits_{n\rightarrow \infty}u_n(x)=0$ for all $x\in V.$ Hence for any $\epsilon>0,$ we can take $n$ sufficiently large such that $\vert u_n(x)\vert\leq\epsilon$ and $\vert u_n(y)\vert\leq\epsilon$ for all $(x,y)\in E^{'}$. Also, we take $n$ sufficiently large such that $\int_G\vert\nabla_G u\vert^pdx\leq J^a_G+\epsilon$. Then we have
    \begin{align*}
        J^{a,p,q}_{\mathbb{Z}^d}&\leq \int_{\mathbb{Z}^d}\vert\nabla_{\mathbb{Z}^d} u_n\vert^p_pdx\\
        &=\int_G\vert\nabla_G u\vert^p_pdx+\frac{1}{2}\sum_{(x,y)\in\mathcal{E}^{'}}\vert u_n(x)-u_n(y)\vert^p\\
        &\leq J^{a,p,q}_G+C\epsilon.
    \end{align*}
    Taking $\epsilon$ sufficiently small contradict $J^{a,p,q}_G< J^{a,p,q}_{\mathbb{Z}^d}.$\\

    If $0<\Vert u\Vert_{l^q(G)}^q< a$, set $\beta=\Vert u\Vert_{l^q(G)}^q$. By Lemma \ref{NormBreLieb} we have 
    \begin{align*}
        &\Vert u_n-u\Vert_{l^q(G)}^q=\Vert u_n\Vert_{l^q(G)}^q-\Vert u\Vert _{l^q(G)}^q + o(1),\\
        &\Vert u_n-u\Vert_{D^{1,p}(G)}^p=\Vert u_n\Vert_{D^{1,p}(G)}^p-\Vert u\Vert _{D^{1,p}(G)}^p + o(1).
    \end{align*}
    Hence we have
    \begin{align*}
        J^{\beta,p,q}_G&\leq \int_{G}\vert\nabla_G u\vert^p_pdx\\
        &=\int_{G}\vert\nabla_G u_n\vert^p_pdx-\int_{G}\vert\nabla_G( u_n-u)\vert^p_pdx+o(1)\\
        &\leq J^{a,p,q}_{G} - J^{a-\beta,p,q}_{G}+o(1).
    \end{align*}
    Taking $n\rightarrow\infty$ we have $ J^{\beta,p,q}_G+J^{a-\beta,p,q}_{G}\leq J^{a,p,q}_{G}$, this contradicts (2) in Lemma \ref{sobolevfunproperty}.\\
    
    The claim is proved. Hence $\Vert u\Vert_{l^q(G)}^q=a$. Also, by the weak lower semi-continuity of the norm $D^{1,p}$, we have $E^{a}_G\leq\int_{G}\vert\nabla_G u\vert^p_pdx\leq \lim\limits_{n\rightarrow\infty}\int_{G}\vert\nabla_G u_n\vert^p_pdx=E^a_G$. Hence, $E^a_G$ is attained by $u$.
\end{proof}
Using Lemma \ref{lemma:sobolevcompare}, we can prove (2) in Theorem \ref{Thm 1}.
\begin{theorem}
    If $q>\frac{dp}{d-p}$, then for all $a>0$, $J^{a,p,q}_{G}$ is attained.
\end{theorem}
\begin{proof}
    It is known that for all $a>0$, $J^{a,p,q}_{\mathbb{Z}^d}$ is attained. We claim that $J^{a,p,q}_{G}<J^{a,p,q}_{\mathbb{Z}^d}$. \\

    Suppose $J^{a,p,q}_{\mathbb{Z}^d}$ is attained by $u$. Take an edge $(y,z)\in E^{'}$. Without loss of generality, we assume $z=y+e_1.$ We claim that there exists $x_0\in \mathbb{Z}^d$ such that $u(x_0)\neq u(x_0+e_1)$. If this not true, then for all $x\in\mathbb{Z}^d$ we have $u(x)=u(x+e_1)$. Hence $u(x)=u(x+ke_1)$ for all $k\in\mathbb{Z}$. Since $u\in l^q(\mathbb{Z}^d)$, $\lim\limits_{\vert x\vert_{\infty}\rightarrow\infty}u(x)=0$. Letting $k\rightarrow\infty$, we have $u(x)=0$ for all $x\in\mathbb{Z}^d$. This is a contradiction. Hence there exists $x_0\in \mathbb{Z}^d$ such that $u(x_0)\neq u(x_0+e_1)$. Set $\tilde{u}(x)=u(x+x_0-y)$, then we have
    \[
    \vert \tilde{u}(y)-\tilde{u}(z)\vert=\vert u(x_0)-u(x_0+e_1)\vert\neq 0.
    \]
    Hence we have
    \begin{align*}
        J^{a,p,q}_G&\leq\int_{G}\vert \nabla_G \tilde{u}\vert^p_pdx \\
        &\leq \int_{\mathbb{Z}^d}\vert \nabla_{\mathbb{Z}^d} \tilde{u}\vert^p_pdx - \vert\tilde{u}(y)-\tilde{u}(z)\vert^p\\
        &< \int_{\mathbb{Z}^d}\vert \nabla_{\mathbb{Z}^d} \tilde{u}\vert^p_pdx\\
        &=J^{a,p,q}_{\mathbb{Z}^d}.
    \end{align*}
    This proves the claim. Hence, by Lemma \ref{lemma:sobolevcompare}, $J^a_G$ can be attained for all $a>0$.
\end{proof}

In the end of this section, we prove (3) in Theorem \ref{Thm 1}.
\begin{theorem}
    If $q=\frac{dp}{d-p}$ and for some $a_0>0$, $J^{a,p,q}_G<J^{a,p,q}_{\mathbb{Z}^d}$, then for all $a>0$ $J^{a,p,q}_G$ can be attained. Moreover, there exists $G$ satisfying conditions.
\end{theorem}
\begin{proof}
    Since $J^{a,p,q}_{G}=(S^{a,p,q}_Ga^{\frac{1}{q}})^{p}$ and $J^{a,p,q}_{\mathbb{Z}^d}=(S^{a,p,q}_{\mathbb{Z}^d}a^{\frac{1}{q}})^{p}$, we have $S^{a_0,p,q}_G<S^{a_0,p,q}_{\mathbb{Z}^d}$. Hence, we have $J^{a,p,q}_{G}<J^{a,p,q}_{\mathbb{Z}^d}$ for all $a>0$. By Lemma \ref{lemma:sobolevcompare}, $J^{a,p,q}_{G}$ can be attained. Now we construct a graph $G$ satisfying conditions. Arbitrarily choose a edge $e_R\in \partial B_R$, set $E^{'}=\partial B_R\backslash \{e_R\}$ and $G=(V,E\backslash E^{'})$. For $a=1$, consider function $f_R=\frac{1}{\vert  B_R\vert^{\frac{1}{q}}}\mathbf{1}_{B_R}$, then $\Vert f_R\Vert_{l^q(G)}^q=1$ and $\Vert f_R\Vert_{D^{1,p}}^p=\frac{1}{\vert  B_R\vert^{\frac{p}{q}}}$. Hence we have $J^{a,p,q}_G\leq\frac{1}{\vert  B_R\vert^{\frac{p}{q}}}$. Choose $R$ sufficiently large such that $\frac{1}{\vert  B_R\vert^{\frac{p}{q}}}<J^{a,p,q}_{\mathbb{Z}^d}$, then we have the desired graph.
\end{proof}
\section{perturbation by adding edges}\label{sec:addingedges}
Let $G=(V,E\cup E^{'})$ be a graph obtained by adding edges $E^{'}$ in lattice graphs $\mathbb{Z}^d$ within the ball $B_R$. When we consider the attainability of $J^{a,p,q}_G$, we assume $d\geq 3$ further. First we show that $E^a_G\leq E^a_{\mathbb{Z}^d}$.
\begin{lemma}\label{lemma:addcompare}
    For all $a\textgreater0$, we have $E^{a,p}_G\leq E^{a,p}_{\mathbb{Z}^d}$. 
\end{lemma}
\begin{proof}
    For any $\epsilon\textgreater0$, take $u\in l^2(\mathbb{Z}^d)$ such that $\Phi^{p}_{\mathbb{Z}^d}(u)\leq E^{a,p}_{\mathbb{Z}^d}+\epsilon$. Take $R_1>0$ sufficiently large such that $\vert u(x)\vert\leq\epsilon$ for all $x\in B_{R_1}^c$. Set $\tilde{u}(x)=u(x+3R_1e_1)$, then we have
    \begin{align*}
        E^{a,p}_{G}&\leq\Phi^p_G(\tilde{u})\\
        &=\Phi^p_{\mathbb{Z}^d}(\tilde{u})+\sum\limits_{(x,y)\in E^{'}}\frac{1}{2}\vert \tilde{u}(x)-\tilde{u}(y)\vert^2\\
        &\leq E^{a,p}_{\mathbb{Z}^d}+\epsilon+C\epsilon^2.
    \end{align*}
    Since $\epsilon$ is arbitrary, we have the desired result.
\end{proof}
Also it is obvious that $E^{a,p}_{\mathbb{Z}^d}\leq E^{a,p}_{G}$, we have $E^{a,p}_{G}=E^{a,p}_{\mathbb{Z}^d}$. 
Now we are ready to prove (1) in Theorem \ref{Thm 2}.
\begin{theorem}
    There exists $E^{'}$ and $a_1>0$ such that $E^{a,p}_G$ can not be attained for all $a>a_1$.
\end{theorem}
\begin{proof}
    We first take $a_1$ sufficiently large such that $E^{a,p}_{\mathbb{Z}^d}<0$ for all $a\geq a_1.$ Let $R>0$ be sufficiently large, to be specified later. We add edges $(0,y)$ if $y\in B_R$ and $(0,y)\notin E$. Also, we add edges $(\pm e_k,y)$ if $y\in B_R$ and $(\pm e_k,y)\notin E$. In the contrary, if $E^{a,p}_G$ can be attained by $u\in l^2(G)$. Since $\int_{G}\vert\nabla_G \vert u\vert\vert^p_p dx\leq\int_{G}\vert\nabla_G  u\vert^p_p dx,$ without loss of generality we assume $u\geq0$. Then there exists $\lambda_a$ such that 
    \begin{equation}\label{normalizedlagrange}
    -\Delta_G u+\lambda_a u= u^{p-1}  .  
    \end{equation}
    
    We claim that $u>0$. If this is not true. Suppose $x_0\in V$ such that $u(x_0)=0$. Then by equation (\ref{normalizedlagrange}) we have $\Delta_{G}u(x_0)=\sum\limits_{y\sim x}u(y)=0.$ Since $u(y)\geq 0$ for all $y\sim x_0,$ we have $u(y)=0$ for all $y\sim x.$ Repeating this process we have that $u=0$. This is a contradiction, and we prove the claim. \\
    
    Since $\Phi_G^p(u)=E^{a,p}_G=E^{a,p}_{\mathbb{Z}^d}$, we have $u(x)=u(y)$ for any $(x,y)\in E^{'}$. Particularly, $u(x)=u(0)$ for all $x\sim 0$. To see this, if $x\neq \pm e_k$, then $(x,0)\in E^{'}$, if $x=\pm e_k$, then we can find $y\in B_R$ such that $(y,\pm e_k)\in E^{'}$ and $(y,0)\in E^{'}$, hence $u(0)=u(y)=u(\pm e_k).$ Hence $\Delta_G u(0)=\sum\limits_{x\sim 0}(u(x)-u(0))=0$ and by equation (\ref{normalizedlagrange}) we have
    \[
    \lambda_a= u^{p-2}(0).
    \]
    One the one hand, multiply by $u$ on the both side of equation (\ref{normalizedlagrange}) and integrate by parts, we have
    \[
    \int_{G}\vert\nabla_G u\vert^2_2dx+\lambda_a\int_{G}u^2dx-\int_{G}\vert u\vert^pdx=0.
    \]
    But then we have 
    \begin{align*}
    \Phi_G^p(u)&=\frac{1}{2}\int_{G}\vert\nabla_G u\vert^2_2dx-\frac{1}{p}\int_{G}\vert u\vert^pdx\\
    &=(\frac{1}{2}-\frac{1}{p})\int_{G}\vert\nabla_G u\vert^2_2dx-\frac{1}{p}\lambda_a\int_{G}u^2dx\\    
    &=(\frac{1}{2}-\frac{1}{p})\int_{G}\vert\nabla_G u\vert^2_2dx-\frac{a}{p}\lambda_a\\
    &\geq -\frac{a}{p}\lambda_a.
    \end{align*}
    Since $\lambda_a= u^{p-2}(0)$ and $\vert B_R\vert u^2(0)\leq \int_{\mathbb{Z}^d}u^2dx=a$, we have
    \[
    E^{a,p}_{\mathbb{Z}^d}=E^{a,p}_{G}=\Phi_G^p(u)\geq -\frac{a}{p}\lambda_a\geq-\frac{1}{p\vert B_R\vert^{\frac{p-2}{2}}}a^{\frac{p}{2}}.
    \]
    On the other hand. Take function $u=\sqrt{a}\delta_0$, then $\Vert u\Vert_{l^2(\mathbb{Z}^d)}^2=a$. we have
    \[
    E^{a,p}_{\mathbb{Z}^d}\leq\Phi^p_{\mathbb{Z}^d}(u) = da-\frac{1}{p}a^{\frac{p}{2}}\leq -Ca^{\frac{p}{2}}
    \]
    for all $a\geq a_1$, where $a_1$ is sufficiently large. Hence taking $R$ sufficiently large such that $-\frac{1}{p\vert B_R\vert^{\frac{p-2}{2}}}>-C$ leads to a contradiction. \\

\end{proof}
Next, we move to the proof for (2) in Theorem \ref{Thm 2}. Similar to Lemma \ref{lemma:addcompare}, we need the following lemma.
\begin{lemma}
    For all $a>0,$ $ J^{a,p,q}_{G}\leq J^{a,p,q}_{\mathbb{Z}^d}$.
\end{lemma}
\begin{proof}
    For any $\epsilon>0$, let $u\in l^q(\mathbb{Z}^d)$ such that $\Vert u\Vert_{l^q(\mathbb{Z}^d)}^q=a$ and $ \int_{\mathbb{Z}^d}\vert\nabla_{\mathbb{Z}^d} u\vert^p_pdx\leq J^{a,p,q}_{\mathbb{Z}^d}+\epsilon$. Take $R_1>0$ sufficiently large such that $\vert u(x)\vert<\epsilon$ for all $x\in B_{R_1}^c$. Let $\tilde{u}(x)=u(x+3R_1e_1)$, then we have
    \begin{align*}
        J^{a,p,q}_{G}&\leq\int_{G}\vert\nabla_{G} \tilde{u}\vert^p_pdx\\
        &=\int_{\mathbb{Z}^d}\vert\nabla_{\mathbb{Z}^d} \tilde{u}\vert^p_pdx+\sum\limits_{(x,y)\in E^{'}}\vert \tilde{u}(x)-\tilde{u}(y)\vert^p\\
        &\leq \int_{\mathbb{Z}^d}\vert\nabla_{\mathbb{Z}^d} \tilde{u}\vert^p_pdx+C\epsilon\\
        &\leq J^{a,p,q}_{\mathbb{Z}^d}+C\epsilon.
    \end{align*}
    Since $\epsilon>0$ is arbitrary, we have the result.
\end{proof}
It is also obvious that $J^{a,p,q}_{\mathbb{Z}^d}\leq J^{a,p,q}_{G}$, we have $J^{a,p,q}_{\mathbb{Z}^d}=J^{a,p,q}_{G}$. Now we can prove (2) in Theorem \ref{Thm 2}.
\begin{theorem}
    There exists $E^{'}$ such that $J^{a,p,q}_G$ can not be attained for all $a>0$.
\end{theorem}
\begin{proof}
     We take $R>10.$ We add edges $(0,y)$ if $y\in B_R$ and $(0,y)\notin E$. Also, we add edges $(\pm e_k,y)$ if $y\in B_R$ and $(\pm e_k,y)\notin E$. If $J^{a,p,q}_G$ can be attained. Suppose $J^{a,p,q}_G$ is attained by $u$. Since $\int_{G}\vert\nabla_G \vert u\vert\vert^p_p dx\leq\int_{G}\vert\nabla_G  u\vert^p_p dx,$ we assume $u\geq0$ without loss of generality. Then for some Lagrange multiplier $\lambda\in\mathbb{R}$ we have
     \begin{equation}\label{equa:sobolevlagrangrain}
     -\Delta_pu = \lambda u^{q-1},    
     \end{equation}
     where $\Delta_p$ is $p$-Laplacian defined as follows.
     \[
     \Delta_p u(x)=\sum\limits_{y\sim x_0}\vert u(y)-u(x)\vert^{p-2}(u(y)-u(x))
     \]
     for $u\in C(G)$. Multiply by $u$ on both side of equation (\ref{equa:sobolevlagrangrain}) and integrate by parts, we have
     \[
     \int_{G}\vert\nabla_G u\vert^p_pdx=\lambda\int_{G}u^qdx.
     \]
     Hence $\lambda\neq 0$. We claim that $u>0$. If this is not true, suppose $x_0\in V$ such that $u(x_0)=0$, then by equation (\ref{equa:sobolevlagrangrain}) we have $\Delta_p u(x_0)=\sum\limits_{y\sim x_0}\vert u(y)\vert^{p-2}u(y)=0.$ But $u(y)\geq0$ for all $y\sim x_0$. We have $u(y)=0$ for all $y\sim x$. Repeating this process, we have $u=0$, which is a contradiction, and we prove the claim. Since $J^{a,p,q}_G=J^{a,p,q}_{\mathbb{Z}^d}$, we have $u(x)=u(y)$ for all $(x,y)\in E^{'}.$ Hence $-\Delta_G u(0)=0$, which implies $\lambda=0$ by equation (\ref{equa:sobolevlagrangrain}) since $u(0)\neq0$. This is a contradiction.              
\end{proof}
\textbf{Acknowledgements.}
The author is deeply grateful to Prof. Bobo Hua for raising this interesting problem, engaging in insightful discussions, offering constructive feedback, and providing unwavering support.


\end{document}